\documentclass[11pt]{amsart}
 \usepackage{graphicx}
 \usepackage{amssymb}
 \usepackage{amsmath}
 \usepackage{amsthm}
 \usepackage{mathtools}
 \usepackage{tikz}
 \usepackage{psfrag}
 \usepackage{xcolor}

\newtheorem{thm}{Theorem}[section]
\newtheorem{lem}[thm]{Lemma}

\newtheorem{cor}[thm]{Corollary}

\theoremstyle{definition}

\newtheorem{remk}{Remark}[section]

\newcommand{\M}{\mathbb M}
\newcommand{\MM}{\mathbb M^{2\times 2}}

\newcommand{\R}{\mathbb R}
\newcommand{\dist}{\operatorname{dist}}

\newcommand{\K}{\mathcal K}

\newcommand{\tr}{\operatorname{tr}}

\newcommand{\rad}{\operatorname{rad}}

\newcommand{\dv}{\operatorname{div}}

\makeatletter
\numberwithin{equation}{section}
\makeatother

\begin{document}

 \author{Baisheng Yan}
 \address{Department of Mathematics\\ Michigan State University\\ East Lansing, MI 48824, USA}
   \email{yanb@msu.edu}

\title[Uniqueness and regularity for polyconvex gradient flows]{On  nonuniqueness and   nonregularity for gradient flows of polyconvex functionals}


\begin{abstract}
We provide some counterexamples  concerning   the uniqueness and  regularity of weak solutions  to the initial-boundary value problem for  gradient flows  of certain strongly polyconvex functionals by showing  that such a problem can possess a trivial classical solution as well as infinitely many  weak solutions that are nowhere smooth. Such polyconvex functions have been constructed in the previous work, and  the nonuniqueness and nonregularity will be achieved by reformulating the gradient flow as a   partial differential relation  and then using the convex integration method to construct certain strongly convergent sequences of subsolutions that have a uniform control on the local essential oscillations of their spatial gradients.
  \end{abstract}

\maketitle

\section{Introduction}

Let  $m,n\ge 2$ be integers,  $\M^{m\times n}$ be the usual Euclidean space of $m\times n$ matrices with  inner product $A:B=\tr(A^TB),$ and $F\colon \M^{m\times n}\to \R$ be a given continuous function. We consider  the energy functional
 \begin{equation}\label{energy}
 \mathcal E(u)=\int_\Omega F(Du(x))\,dx  
 \end{equation}
for  $u=(u^1,\cdots, u^m)\colon\Omega\to \R^m,$ where $\Omega$ is a bounded domain in $\R^n$ and   $Du$  is the Jacobian matrix of $u$  in $\M^{m\times n}$ defined  by $Du =({\partial u^i}/{\partial x_k}) \, (1\le i\le m,\, 1\le k\le n).$ 

 Assume $F$ is $C^1$ and define $DF\colon \M^{m\times n}\to \M^{m\times n}$ by setting  $DF(A)=\big( {\partial F(A)}/{\partial A^i_k} \big)$ for all $A\in \M^{m\times n}.$ Given a number $T>0$, the {\em $L^2$-gradient flow} of energy functional $\mathcal E$ on $\Omega_T=\Omega\times (0,T)$ is defined by  
\begin{equation}\label{GF}
\partial_t u=\dv DF(Du),
\end{equation}
where  $\partial_t u$ and $Du$ denote the time-derivative and spatial Jacobian matrix of $u=u(x,t)\colon\Omega_T\to\R^m.$ 

Smooth solutions of the gradient flow (\ref{GF}) with time-independent  boundary conditions $u(x,t)=g(x)$ on $\partial\Omega\times (0,T)$  will obey the energy decay identity
\[
\frac{d}{dt}\mathcal E(u(\cdot,t))=-\int_\Omega |\partial_t u|^2 dx
\]
and thus will always decrease the energy. In this regard,  gradient flows  become  important as well as an ideal time-dependent evolution model  in the study of energy minimization both theoretically and in applications; see, e.g.,  {\sc  Ambrosio, Gigli \&  Savar\'e} \cite{AGS}. As for general systems of  partial differential equations,  the existence of  smooth solutions of  gradient flows may not be  always  available and thus the study of certain solutions in some weaker sense becomes necessary, in hoping that such solutions can still retain some features of the smooth solutions. A function $u\colon \Omega_T\to  \R^m$ in a usual Sobolev space is called   {\em   a weak solution} of (\ref{GF}) provided  that
\begin{equation}\label{weak-sol}
\int_\Omega  (u  \cdot \varphi) \big |_{t=0}^{t=T} \, dx   =\int_{\Omega_T} \Big (u \cdot \partial_t \varphi  - DF(Du) : D\varphi \Big ) dxdt 
\end{equation}
holds for all $\varphi\in C^1([0,T]; C^\infty_c(\Omega;\R^m)),$ where $u(\cdot,0)$ and $u(\cdot,T)$ are taken in the sense of trace.
However, some unexpected counterexamples to the features enjoyed by smooth solutions may occur for  weak solutions.  The goal of this paper is to provide such  counterexamples for the gradient flows  of  certain  energy functionals for which the energy minimization problems have traditionally been  well-studied; see, e.g.,  {\sc Ball} \cite{Ba} and {\sc Dacoronga} \cite{D}. 

 In this paper, we study the gradient flow  (\ref{GF}) for certain smooth {\em strongly polyconvex} functions  $F\colon \MM\to \R$  in the sense that for some $\nu>0$ 
\begin{equation}\label{poly}
F(A)=\frac{\nu}{2}|A|^2+G(A,\det A),
\end{equation}
 where $G\colon \MM\times \R\to \R$  is convex. Our main  concern is about  the uniqueness and regularity of  weak solutions to the initial-boundary value problem
\begin{equation}\label{ibvp}
\begin{cases}  \partial_t u=\dv DF(Du )   \;\; \mbox{in $\Omega_T$;}
\\
 u  |_{\partial' \Omega_T}=0,
\end{cases}
\end{equation}
 where $\partial'\Omega_T=(\Omega\times \{0\})\cup (\partial \Omega \times (0,T))$ denotes  the parabolic boundary of $\Omega_T.$ Clearly, the energy decay  identity shows that $u\equiv 0$ is the  unique smooth solution to (\ref{ibvp}), but for weak solutions, this may not be the case;  we establish the following main result  concerning the nonuniqueness and  nonregularity  of (\ref{ibvp}).

\begin{thm}\label{mainthm}   There exist  smooth strongly polyconvex functions $F$ of the form (\ref{poly}) on $\MM$  with the property that for all $\phi\in C^1_0(\Omega;\R^2)$ there is $\epsilon_0>0$ such that for all  $|\epsilon|<\epsilon_0$  the initial-boundary value problem (\ref{ibvp})  possesses infinitely many  Lipschitz weak solutions  $u$ on $\Omega_T$ satisfying that $u(x,T)=\epsilon \phi(x)$ and that $Du$ is nowhere  continuous in $\Omega_T.$ 
  \end{thm}

The polyconvex functions $F$ in the theorem will be  the same as constructed in {\sc Yan} \cite{Ya1}, and the theorem will be proved as a corollary  of a more general result  proved later (see Theorem \ref{mainthm1}).  Some aspects of the theorem in regard to certain known results will be discussed below,  along with  the main strategies of the proof.

In general, when studying the energy functional  (\ref{energy}) or a general diffusion system  of the form
\begin{equation}\label{sys0}
\partial_t u=\dv \sigma(Du) 
\end{equation}
with diffusion-flux  function $\sigma\colon \M^{m\times n}\to \M^{m\times n}$  given and continuous, one often imposes certain convexity or monotonicity conditions on the function  $F$ or  $\sigma;$ see  \cite{AF, Ba,  BDM13, CZ92, D, Ev, Fu87, Gi, Ha95, La96, Mo, Zh86}.  Under certain such conditions, some {\em partial regularity} results have been well-known. For example, it has been proved by {\sc Evans} \cite{Ev}  that all Lipschitz  minimizers of  $\mathcal E$ are $C^{1,\alpha}$ on some  open subset of $\Omega$ of full measure provided that  the function $F$ is  {\em strongly quasiconvex} in the sense that for some $\nu>0$ 
\begin{equation}\label{qx}
 \int_\Omega (F(A+D\phi)-F(A))dx \ge  \frac{\nu}{2} \int_\Omega |D\phi|^2 dx 
  \end{equation}
  holds for all $A\in\M^{m\times n}$ and $\phi\in C_c^\infty(\Omega;\R^m).$  Concerning the diffusion system (\ref{sys0}),  under some reasonable smoothness and growth conditions on $\sigma$, it has been  proved by {\sc  B\"ogelein,  Duzaar \& Mingione} \cite{BDM13} that any weak solution  $u$  of  (\ref{sys0}) satisfies  $Du\in C_{loc}^{\alpha,\alpha/2}(Q_0)$ for some $\alpha\in (0,1)$ on some  open subset $Q_0$ of $\Omega_T$ of full measure provided that the function $\sigma$ is  {\em strongly quasimonotone} in the sense that for some $\nu>0$  \begin{equation}\label{qm}
\int_\Omega  \sigma (A+D\phi) : D\phi \, dx \ge \nu\int_\Omega |D\phi|^2 dx 
\end{equation}
  holds for all $A\in\M^{m\times n}$ and $\phi\in C_c^\infty(\Omega;\R^m).$ Note that by approximation both (\ref{qx}) and (\ref{qm}) will hold for all $\phi\in W^{1,\infty}_0(\Omega;\R^m)$ if they hold for all $\phi\in C_c^\infty(\Omega;\R^m).$

It is easily seen that a  function of the form (\ref{poly}) always satisfies  (\ref{qx}) and that  a  function  $F$  will satisfy (\ref{qx}) if $\sigma=DF$ satisfies (\ref{qm}). The result of \cite{BDM13}  shows that $\sigma=DF$ cannot be  strongly quasimonotone  for any   function  $F$  in  the theorem.   In general, unlike polyconvexity implying quasiconvexity,  it seems impossible to characterize an easily-recognizable class of  functions $F$ such that  $\sigma=DF$ is (strongly) quasimonotone. 
As a consequence of  (\ref{qx}) and the nonregularity, we also have the following result.
\begin{cor}
Let  $u$ be any weak solution to the problem (\ref{ibvp}) as in Theorem \ref{mainthm}. Then there exists an open dense subset $I$ of  $(0,T)$ such that  
\begin{equation}\label{energy-0}
\mathcal E(u(\cdot,t))-\mathcal E(u(\cdot,0))\ge  \frac{\nu}{2} \int_\Omega |Du(x,t)|^2 dx>0 \quad \forall \, t\in I;
\end{equation}
 thus the  energy does not decrease  along the solution. 
 \end{cor}
 \begin{proof} The set
 $C=\{t\in (0,T) \,|\, u(\cdot,t)|_{\Omega}=0\}$ is closed in $(0,T)$ and has no interior; otherwise, $u\equiv 0$ and thus $Du$ is continuous on an open subset of $\Omega_T.$  So $I=(0,T)\setminus C$ is open and dense in $(0,T)$, and $\|Du(\cdot,t)\|_{L^2(\Omega)}>0$ for all $t\in I.$ Thus (\ref{energy-0}) follows from the strong quasiconvexity (\ref{qx}) with $A=0$ and $\phi(x)=u(x,t).$
\end{proof}

It has been well-known that partial regularity for the weak solutions of gradient flow  (\ref{GF}) fails under the condition (\ref{qx}) or even (\ref{poly}). Indeed,   examples of   strongly quasiconvex functions $F$ on $\M^{2\times 2}$ have been constructed in  {\sc  M\"uller \& \v Sver\'ak} \cite{MSv2}   such   that the system (\ref{sys0}) with $\sigma=DF$ possesses {\em stationary} (time-independent) Lipschitz weak solutions that are not $C^1$ on any open subset of $\Omega_T.$ A subsequent  result in   {\sc  Sz\'ekelyhidi} \cite{Sz1} has shown that such functions $F$ can be constructed  of the form (\ref{poly}). However, such stationary solutions cannot produce  a counterexample for the initial-boundary value problem (\ref{ibvp}).  Nevertheless, for time-dependent problems, with  any function  $F$ constructed in \cite{MSv2, Sz1},  the result of {\sc  M\"uller,  Rieger \&  \v Sver\'ak} \cite{MRS} shows that for all $\epsilon>0$ and $\alpha\in(0,1)$  there exists  a function $f\in C^\alpha(\Omega_T;\R^2)$ with $\|f\|_{C^\alpha(\Omega_T)}<\epsilon$ such that the initial-boundary value problem:
\begin{equation}\label{ibvp-0}
\begin{cases}  \partial_t u -\dv DF(Du )=f   \;\; \mbox{in $\Omega_T$;}
\\
u|_{\partial'\Omega_T}=0,
\end{cases}
\end{equation}
 has a Lipschitz weak  solution  $u$ with  $Du$ nowhere  continuous in $\Omega_T.$  Their method relies on constructing a parametrized family of weak solutions $u(\cdot;t)$ (with time $t$ being the parameter)   of the stationary system $\dv DF(Du)=0$ on $\Omega$ with  $Du$ nowhere  continuous;  in the meantime, their construction ensures that $\|\partial_t u\|_{C^\alpha(\Omega_T)}<\epsilon;$ finally, their example is provided by taking $f=\partial_t u.$     
 Our result  provides a much simpler counterexample of nonuniqueness and non-partial-regularity for  (\ref{ibvp-0}) with $f=0.$ 
 
A general approach for system (\ref{sys0}) has been set up in {\sc Kim \& Yan} \cite{KY1,KY2,KY3} for some forward-backward diffusion equations for  $n\ge 2$ and $m=1$ and in {\sc Yan} \cite{Ya1} for $n\ge 2$ and $m\ge 1$, following the early work of {\sc Zhang} \cite{Zh1,Zh2} for $m=n=1.$ To discuss the main idea of \cite{Ya1},  we consider functions $u\colon\Omega_T\to\R^m$ and $v\colon\Omega_T\to \M^{m\times n}$ satisfying $u=\dv v$ almost everywhere in $\Omega_T$, where $\dv v$  is the divergence  vector in $ \R^m$  with $(\dv v)^i= \sum_{k=1}^n\partial v^i_k/\partial x_k$ for $i=1,\dots,m.$ The system (\ref{sys0}) can be reformulated under the general framework of {\sc Gromov} \cite{Gr86}  as a  {\em partial differential relation}:
 \begin{equation}\label{pdr}
u= \dv v,\quad (Du,\partial_t v)\in \K_\sigma,
 \end{equation}
where $\K_\sigma=\{(A,\sigma(A))\,|\, A\in \M^{m\times n}\}$ is the graph of $\sigma.$ 
The method of {\em convex integration} for general partial differential relations  (see \cite{Gr86,KMS,MSv,MSv2})   motivates  some assumptions and structures on  the set $\K_\sigma.$ Under such conditions, some dense families of the so-called {\em subsolutions}  of (\ref{pdr}) have been constructed and with them  the existence of true weak solutions  has been established in \cite{Ya1} by  a  Baire  category method (see  \cite{D,DM1}); however, the Baire category method  does not seem to produce  the nowhere-$C^1$ weak solutions. 

In this paper, we achieve the nonuniqueness and nonregularity of problem (\ref{ibvp}) by following the approach  of \cite{MRS,MSv2, Sz1} to construct using subsolutions of  the  relation (\ref{pdr}) with $\sigma=DF$  some Cauchy sequences  $\{u_j\}$  in $W^{1,1}(\Omega_T;\R^2)$ satisfying 
\begin{equation}\label{approx}
u_j=\dv v_j,\quad \|DF(Du_j)-\partial_tv_j  \|_{L^1(\Omega_T)}\to 0
\end{equation}
such that there is a uniform control  on the local essential oscillations of  $\{Du_j\}.$  To accomplish such construction,  the  special features of the polyconvex functions $F$ constructed in \cite{Ya1} need to be greatly explored,  along with some careful and subtle  construction  for  convex integration (see Theorem \ref{thm1});  as in \cite{KY1,KY2,KY3,Ya1}, the linear constraint $u=\dv v$ is handled by a simple antidivergence operator as constructed  in \cite{KY1}  (see Section \ref{s3}).  A more general {\em $h$-principle} type of result   in the sense of \cite{Gr86} is  proved in  Theorem \ref{mainthm1}, from which Theorem \ref{mainthm} follows easily as a corollary.  

Finally, we  remark in passing that the convex integration method  has  recently found remarkable success  in the study of many important partial differential equation problems; see, e.g.,  \cite{BDS,BV,CFG,DS,Is,LP, Sh}.

\section{The polyconvex functions and associated properties}
Smooth polyconvex functions  $F\colon \MM\to \R$ of the form (\ref{poly}) have been constructed in \cite[Section  6]{Ya1} based on a {\em special  class} of  $T_5$-configurations  \cite{KMS,Ta93} in $\MM\times \MM.$ We point out that the strongly polyconvex functions  constructed in \cite{Sz1} are based on regular $T_5$-configurations in $\MM\times \MM$ and thus may not satisfy the special features of $F$ required here.

We assume that $F\colon \MM\to \R$ is any smooth polyconvex function of the form (\ref{poly}) constructed in \cite[Definition 6.1]{Ya1} and denote   the graph of $DF\colon \MM\to\MM$  by \[
\K=\{(A,DF(A))\,|\, A\in \MM\}\subset  \MM\times \MM.
\]
Rather than discussing   the details of the construction, we summarize some properties about $F$ that are crucial for  the proof of the main theorem by convex integration of the partial differential relation
 \begin{equation}\label{pdr1}
u= \dv v,\quad (Du,\partial_t v)\in \K.
 \end{equation}

From the construction, there exist  smooth  functions defined on a closed ball $\bar B_\beta=\bar B_\beta(0) \subset \R^8$,
\begin{equation}
\begin{cases} 
\lambda_i\colon \bar B_\beta  \to (0,1),\\
 \xi_i\colon \bar B_\beta   \to \K,\\
\pi_i\colon \bar B_\beta \to \MM\times \MM,\end{cases} \quad (i=1,\cdots,5)
\end{equation}
satisfying  the following properties:
\begin{itemize}
\item[\bf (P1)] Each $\pi_i$ is invertible and thus $\pi_i(B_{\beta})$ is open for $1\le i\le 5.$ Also $\pi_1(0)=(0,0)\in \MM\times \MM.$ Let $\mathbb P\colon \MM\times \MM\to\MM$ be the projection operator: $\mathbb P(A,B)=A.$ Then
\[
 \mathbb P\xi_i(\bar B_\beta)\cap  \mathbb P\xi_j(\bar B_\beta)= \mathbb P\pi_i(\bar B_\beta)\cap  \mathbb P\pi_j(\bar B_\beta)=\emptyset \quad \forall\, i\ne j.
 \]
\item[\bf (P2)] For  all $1\le i\le 5$ and $q\in \bar B_\beta$,
\[
\pi_{i+1}(q)=\lambda_i(q)\xi_i(q)+(1-\lambda_i(q))\pi_i(q),
\]
where we set $\pi_6(q)=\pi_1(q).$ There exist numbers $0<\nu_0<\nu_1<1$ such that
\[
\nu_0 < \lambda_i(q)<\nu_1 \quad \forall\, 1\le i\le 5,\; q\in \bar B_\beta.
\]

 \item[\bf (P3)] For $0\le \lambda\le 1$ and $1\le i\le 5,$ define the sets in $\MM\times\MM$:
\[
\begin{cases}
S_i(\lambda)   =\{\lambda \xi_i(q) +(1-\lambda)\pi_i(q): q\in B_{\beta}\}, \\
K(\lambda)  =  \cup_{i=1}^5 S_i(\lambda),  \quad \Sigma(0)=K(0), \\
 \Sigma(\lambda) = \cup_{0\le \lambda'<\lambda} K(\lambda') \quad\forall\, 0<\lambda\le 1.
 \end{cases}
\]
Then the following properties hold:
\begin{itemize}
\item[(a)] $S_i(0)=\pi_i(B_\beta)$ is open for each $1\le i\le 5$ and the family $\{S_i(0)\}_{i=1}^5$ is disjoint   (see Property (P1));
\item[(b)]  $\Sigma(\lambda)$ is open for all  $0\le \lambda \le 1$ (see \cite[Proof of Theorem 6.11]{Ya1});
\item[(c)] there is a number  $0<\delta_1<1$ such that
 $\{S_i(\lambda)\}_{i=1}^5 $ is a family of disjoint  open sets   for all  $\delta_1\le \lambda<1 $ (see \cite[Remark 6.3]{Ya1}).
 \end{itemize}
\end{itemize} 
(See Figure \ref{fig0} for a typical special $T_5$-configuration of points $ X_i=\xi_i(q)$ and $ P_i=\pi_i(q) $   in $\MM\times \MM$ for $1\le i\le 5$ with a given $q\in \bar B_\beta.$)

\begin{figure}[ht]
\begin{center}
\begin{tikzpicture}[scale =0.8]
\draw[thick] (-5,-2)--(1,-1);
  \draw(-5,-2) node[above]{$X_1$};
   \draw(1,-1) node[right]{$P_1$};
     \draw[thick] (1,-1)--(1,3);
  \draw[thick] (1,-1)--(1,-3);
  \draw[thick] (-2,-1.5)--(-5,1);
   \draw[thick] (-6.5,2.25)--(-5,1);
   \draw[thick] (-5,1)--(-2,4);
     \draw[thick] (0,6)--(-2,4);
    \draw[thick] (-2,4)--(1,3);
       \draw[thick] (1,3)--(4,2);
     \draw(1,-3) node[right]{$X_5$};
      \draw(1,3) node[above]{$P_5$};
        \draw(-5,1) node[left]{$P_3$};
         \draw(-2,-1.5) node[below]{$P_2$};
               \draw(-2.4,3.8) node[above]{$P_4$};
    \draw(4,2) node[above]{$X_{4}$};
   \draw(-6.5,2.25) node[above]{$X_2$};
     \draw(0,6) node[above]{$X_3$};
       \draw (-5,-2) circle (0.09);
        \draw(4,2) circle (0.09);
         \draw(1,-3) circle (0.09);
          \draw(-6.5,2.25) circle (0.09);
           \draw(0,6) circle (0.09);
\end{tikzpicture}
\end{center}
\caption{The points $ X_i=\xi_i(q) $ and $ P_i=\pi_i(q) $   in $\MM\times \MM$ for $1\le i\le 5$ with a given $q\in \bar B_\beta.$ 
Here $P_{i+1} =\lambda_i X_i +(1-\lambda_i)P_i$ for $1\le i\le 5$ with $P_6=P_1$ and $\lambda_i=\lambda_i(q).$} 
\label{fig0}
\end{figure}
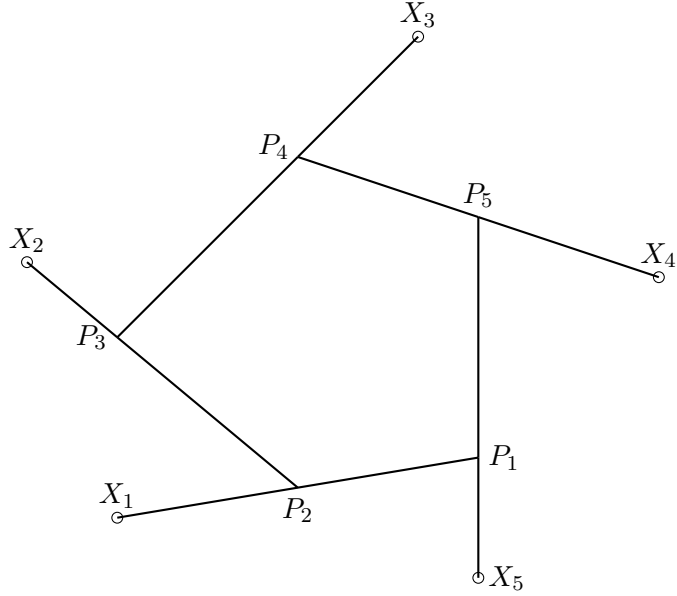
   
   The following elementary lemma, which follows by direct iteration,  clarifies  some computations below;  we omit its proof.
   
\begin{lem}\label{lem0} If quantities $P_k$ and $X_k$, $1\le k\le 5$, satisfy the relation $P_{k+1}=t_kX_k+(1-t_k)P_k$ with  $0<t_k<1$ for all $k$ modulo 5, then 
$
P_i=\sum_{j=1}^5 \nu_{i}^{j}X_{j}$ for all $1\le i\le 5,$ with coefficients $\nu_i^j$ being determined for $1\le i,j\le 5$ by
\[
\nu_{k+1}^k=\frac{t_k}{t},\quad \nu_{k+1}^{k-1}=\frac{(1-t_k)t_{k-1}}{t},\quad  \nu_{k+1}^{k-2} =\frac{(1-t_k)(1-t_{k-1})t_{k-2}}{t},
 \]
 \[
  \nu_{k+1}^{k-3} =\frac{(1-t_k)(1-t_{k-1})(1-t_{k-2})t_{k-3}}{t},
  \]
  \[
\nu_{k+1}^{k-4} =\frac{(1-t_k)(1-t_{k-1})(1-t_{k-2})(1-t_{k-3})t_{k-4}}{t},
\]
for all $k$ modulo 5, where $t=1-(1-t_1)\cdots(1-t_5).$
\end{lem}

A crucial property of the functions $F$ constructed is the following result, which is  the building block for convex integration of the  relation (\ref{pdr1}). 

\begin{thm}\label{lem1} 
Let $Y=(A,B)\in S_i(\lambda)$ for some $1\le i\le 5$ and $0\le \lambda<1.$  Let $
1>\mu'> \mu> \max\{\lambda, \,\nu_1,\,\delta_1\}$ be given numbers,   where $\nu_1$ and $\delta_1$ are the numbers in  (P2) and (P3). Assume  $Y=\lambda \xi_i(q)+(1-\lambda) \pi_i(q)$ for some $q\in B_{\beta}.$  Let
$X_k=\mu\xi_k(q)+(1-\mu) \pi_k(q)\in S_k(\mu)$ for all $1\le k\le 5.$ We have
\[
\pi_{k+1}(q)=\frac{\lambda_k(q)}{\mu}X_k+\left(1-\frac{\lambda_k(q)}{\mu}\right )\pi_k(q) \quad \forall \, k.
\]
By the lemma above,  we have $\pi_k(q)=\sum_{j=1}^5 \tilde \nu_k^j(\mu,q) X_j$ with coefficients $\tilde \nu_k^j(\mu,q)\in (0,1)$  satisfying  
\begin{equation}\label{eq-1}
\sum_{j=1}^5 \tilde \nu_k^j(\mu,q) =1, \quad  \tilde \nu_k^j(\mu,q)  >  (\mu-\nu_1)^4\nu_0 \quad \forall\, 1\le k,\, j\le 5,
\end{equation}
where $\nu_0$ is the number in (P2), and also 
\[
 Y=\left (\frac{\lambda}{\mu} +\big (1-\frac{\lambda}{\mu}\big ) \tilde \nu_i^i(\mu,q)\right )X_i+\sum_{1\le k\le 5,\, k\ne i} \big ( 1-\frac{\lambda}{\mu}\big ) \tilde \nu_i^k(\mu,q) X_k.
 \]
Then, for any bounded open set $G\subset \R^2\times\R$  and  $0<\epsilon<1$,  there exists  $(\varphi, \psi)\in C^\infty_c(\R^2\times\R;\R^2\times\MM)$ with support $\subset\subset G$ such that
\begin{itemize}
\item[(a)] $\dv\psi=0$ and $(A+D\varphi, B+\partial_t \psi)\in \Sigma(\mu')$ on $G;$  
\item[(b)] ${\int_{\R^2} \varphi(x,t)\,dx=0}$ for all $t\in \R;$ 
\item[(c)] $\|\varphi\|_{L^\infty(G)}+\|\partial_t \varphi\|_{L^\infty(G)}<\epsilon;$  
\item[(d)]  there exist disjoint open  subsets $G_1,\cdots,G_5$ of $G$ such that 
\begin{equation}\label{eq-d}
\begin{cases} (A+D\varphi, B+\partial_t\psi) \big |_{G_k} =X_k,\\
|G_i|> (1-\epsilon) \left (\frac{\lambda}{\mu} +(1-\frac{\lambda}{\mu}) \tilde \nu_i^i(\mu,q)\right ) |G|,\\
|G_k|> (1-\epsilon)   ( 1-\frac{\lambda}{\mu}) \tilde \nu_i^k(\mu,q) \,|G| \quad\forall\, k\ne i;
\end{cases} 
\end{equation}
in particular, $\sum_{j=1}^5 |G_j|>(1-\epsilon)|G|.$
\end{itemize}
 \end{thm}
  \begin{proof}  The theorem depends on further special features of $F$ in terms of  functions $\xi_i$ and $\pi_i$ and can be proved  in the same  way as  for the steps 1 and  2 in the proof of \cite[Theorem 5.1]{Ya1}  using  Theorem 2.4  of  the paper \cite{Ya1}; we refer the reader to that paper and omit the details  here. 
  \end{proof}

   \section{A simple antidivergence  operator}\label{s3}
  
Throughout  the rest of the paper,  given a bounded open set $G\subset\R^d$ and a function $f\in W^{1,\infty}(G;\R^q),$ we define  the Dirichlet class
\[
 W^{1,\infty}_f(G;\R^q)=\{g\in W^{1,\infty}(G;\R^q)\,|\;   (g-f)|_{\partial G}=0\}.
 \]

 In this section, we construct a simple antidivergence  operator on  rectangles in $\R^2$ to handle the linear constraint $u=\dv v$  in the relation (\ref{pdr1}). See {\cite[Section 2.4]{KY1}} for  the case of general dimensions and \cite{BB} for some deeper results. 

 Let $J=(a_1,b_1)\times (a_2,b_2)$ be a rectangle in $\R^2$. For any function $u=(u^1,u^2)\colon J \to \R^2,$ we define $v=\mathcal R^J u\colon  J\to \MM$ by setting  $v=(v^i_k(x))$  for  all $x=(x_1,x_2)\in J,$ with
 \begin{equation}
\begin{split} v^i_1(x)&=\rho(x_2) \int_{a_1}^{x_1}\left (\int_{a_2}^{b_2} u^i(r,s)ds\right )dr,\\
 v^i_2(x)&=\int_{a_2}^{x_2} u^i(x_1,s)\,ds - \int_{a_2}^{x_2}  \rho(s)\,ds  \int_{a_2}^{b_2} u^i(x_1,s)ds 
 \end{split} \quad (i=1,2),
 \end{equation}
  where $\rho\in C^\infty_c(a_2,b_2)$ is a fixed function with $\int_{a_2}^{b_2}\rho(s)ds=1.$ 
 
 We easily verify the following result.
 
 \begin{lem} \label{anti-div} The operator  $\mathcal R^J \colon C(\bar J;\R^2)\to C(\bar J;\MM)$  is  a bounded linear operator and satisfies $\dv v=u$ in $J.$  Moreover, if $u\in W^{1,\infty}(J;\R^2)$ then $v=\mathcal R^J u\in W^{1,\infty}(J;\MM);$ if $u\in W^{1,\infty}_0(J;\R^2)$ satisfies  $\int_J u(x)dx=0$, then  $v=\mathcal R^J u\in W_0^{1,\infty}(J;\MM).$ Furthermore, if $u\in C^1(\bar J;\R^2)$ then $v=\mathcal R^J u$ is in $C^1(\bar J;\MM).$ 
\end{lem}
 
 \begin{remk} Other antidivergence operators $\mathcal A$ may be defined through certain unique solutions of  Poisson's equation $\Delta f=u$ on $J$ by setting  $v=\mathcal Au:=Df.$ Although they may have the $L^\infty$-$L^\infty$ bound on $W^{1,\infty}$ functions,  they don't satisfy the boundary property that  $v=\mathcal A u\in W_0^{1,\infty}(J;\MM)$ for all $u\in W^{1,\infty}_0(J;\R^2)$ satisfying $\int_J u(x)dx=0.$
 \end{remk}
 
Let $Q_0=(-1,1)^2$ be the cube in $\R^2$ and $Q=Q_0\times (-1,1)$ be the cube in $\R^2\times \R=\R^3.$ 
 
 \begin{lem}\label{div-inv}  Let  $u\in W^{1,\infty}_0(Q;\R^2)$ satisfy $\int_{Q_0} u (x,t)\,dx=0$ for all $t\in (-1,1).$  Then there exists a function $ v =\mathcal R  u\in W^{1,\infty}_0(Q;\MM)$ such that $\dv  v =u$ in $Q$ and
\begin{equation}\label{div-1}
\| \partial_t v\|_{L^\infty(Q)}  \le C_0  \|\partial_t u\|_{L^\infty(Q)},
\end{equation}
where $C_0$ is a constant. Moreover, if in addition $u\in C^1(\bar Q;\R^2)$ then $ v =\mathcal R  u \in C^1(\bar Q;\MM).$
\end{lem}
\begin{proof} Here we simply take $\mathcal Ru=\mathcal R^{Q_0}u(\cdot,t).$ The estimate (\ref{div-1}) follows from $\partial_t(\mathcal Ru)=\mathcal R^{Q_0}(\partial_t u(\cdot,t)).$ 
\end{proof}

For $y=(x,t)\in \R^2\times \R$ and $l>0,$ let $Q_{y,l}=y+lQ$  be the cube of center $y$ and side length $2l.$ For convenience,  we denote  the number $ l=\rad(Q_{y,l})$ as a radius  of $Q_{y,l},$ with a possible  abuse of notation.

Define the {\em rescaling operator} $\mathcal L_{y,l}  \colon W^{1,\infty}(Q;\R^d)\to W^{1,\infty}(Q_{y, l};\R^d)$ by  setting
\begin{equation}\label{scaling}
(\mathcal L_{y,l} f)(z)=lf(\frac{z- y}{l})\quad (z\in Q_{y,l}) \quad \forall \,f\in W^{1,\infty}(Q;\R^d).
\end{equation}
 Note that $\nabla (\mathcal L_{y,l} f)(z)=(\nabla f)(\frac{z- y}{l})$ for $z\in Q_{y, l},$ where $\nabla=(D_x,\partial_t)$ if $z=(x,t)\in \R^2\times \R.$
From the previous lemma, the following result is immediate.
 
 \begin{cor}\label{div-inv-1}  Let  $\varphi\in W^{1,\infty}_0(Q;\R^2)$ satisfy  
$
\int_{Q_0} \varphi (x,t)\,dx=0$ for all $t\in (-1,1).
$
  Let $\tilde \varphi =\mathcal L_{\bar y,l}\varphi$ and 
 $\tilde g=  \mathcal R_{y, l} \varphi:=l \mathcal L_{y,l} (\mathcal R \varphi)$ in $W^{1,\infty}_0(Q_{y, l};\MM).$
  Then 
\begin{equation}\label{div-2}
\mbox{$\dv \tilde g=\tilde \varphi$  in $Q_{y,l}$,} \quad \|\partial_t \tilde g \|_{L^\infty(Q_{y, l})}  \le C_0 l  \|\partial_t \tilde \varphi\|_{L^\infty(Q_{y, l})}.
\end{equation}
Moreover, if in addition $\varphi\in C^1(\bar Q;\R^2)$ then $\tilde g =\mathcal R_{y, l} \varphi \in C^1(\bar Q_{y, l};\MM).$
\end{cor}

   
   \section{Convex integration and the main construction}

In this section, we provide  the main crucial construction  for convex integration of the partial differential relation (\ref{pdr1}).

\begin{thm}\label{thm1} Let $G\subset \R^2\times \R$ be a bounded open set and let $(u,v)\in C^1(\bar G;\R^2\times\MM)$ satisfy  
\begin{equation}\label{strict-sub}
u=\dv v,\quad (Du,\partial_t v)\in \Sigma(\lambda') \;\; \mbox{on $\bar G$}
\end{equation}
 for some $0<\lambda'<1.$  Let  $ 0\le \lambda<  \lambda'<\mu<\mu'<1$ with $\mu> \max\{ \nu_1, \delta_1\}$ be given numbers such that 
 \begin{equation}\label{open}
 \mbox{$\{S_k(\lambda)\}_{k=1}^5$ is a  family of disjoint open sets.}
 \end{equation}
  Then, for all $0<\epsilon<1$,   there exists a function $(\tilde u,\tilde v)\in W^{1,\infty}_{(u,v)}(G;\R^2\times \MM)$ that is piece-wise $C^1$ on at most countable partition of $G$ by disjoint closed cubes $\{\bar Q_j\}_{j=1}^M$  $(M\le \infty)$  such that
\begin{itemize}
\item[(a)] $0<\rad(Q_j)<\epsilon$ for all $j;$   
 \item[(b)] $\tilde u=\dv \tilde v, \; (D\tilde u,  \partial_t \tilde v)\in \Sigma(\mu') $ on each $\bar Q_j;$  
 \item[(c)] $\|\tilde u-u\|_{L^\infty(G)}+ \|\partial_t \tilde u-\partial_t  u\|_{L^\infty(G)}<\epsilon;$
\item[(d)] $|Q_j\cap \{(D\tilde u,\partial_t \tilde v) \in K(\mu) \}| >(1-\epsilon)\,|Q_j|$ for all $j;$
\item[(e)]  for all $1\le k\le 5,$
\[
\begin{cases} 
 |\{(D\tilde u,\partial_t \tilde v) \in S_k(\mu) \}| >\frac12 (\mu- \lambda') (\mu-\nu_1)^4\nu_0\,|G|,\\
 |\{(D\tilde u,\partial_t \tilde v) \in S_k(\mu) \}| >\frac{\lambda}{\mu} |\{(D u,\partial_t  v) \in  S_k(\lambda) \}|; \end{cases}
\]
\item[(f)] if $G_0=G\cap \{(D u,\partial_t  v) \notin K(\lambda)\},$ then 
\[
\|D\tilde u-Du\|_{L^1(G)}\le C\big[ |G_0|+ \big  (\epsilon+(\mu- \lambda)\big ) |G|\big ] ,
\]
where $C>0$  depends only on the polyconvex function $F.$
\end{itemize}  
 \end{thm}
 
 \begin{remk} (i) A function $(u,v)$ satisfying (\ref{strict-sub}) with $0<\lambda'<1$ is called a  (strict) {\em subsolution} of the relation (\ref{pdr1}) on $G$.
 
 (ii) The condition (\ref{open}) seems necessary for  the theorem; nevertheless, by Property (P3), it is always satisfied when $\lambda=0$ or $\lambda\ge \delta_1.$ This will be sufficient  for the proof in the next section.
 \end{remk}

  \begin{proof}[Proof of Theorem \ref{thm1}] Let $G_k= G\cap \{(D u,\partial_t  v) \in S_k(\lambda)\}$ for $1\le k\le 5.$ Then $G_k$ is open for each $1\le k\le 5$ and the family $\{G_k\}_{k=0}^5$ is disjoint and covers $G$, but some $G_k$ may be empty. For each $1\le  k\le 5,$ we select another  open set  $G_k'\subset G_k$  such that
  \begin{equation}\label{set-G'}
  |\partial G_k'|=0,\quad |G_k\setminus G_k'|\le \epsilon'\,|G_k|,
  \end{equation}
where  $0<\epsilon'<\epsilon$ is a sufficiently small number to be determined  later; thus $G_k'\ne \emptyset$ if and only if $G_k\ne \emptyset.$ Define the open set
\[
G'_0= G\setminus \overline{\cup_{k=1}^5  G_k'}= G\setminus \cup_{k=1}^5  \overline{G_k'}.
\]
   Then $\{G'_j\}_{j=0}^5$ is a family of disjoint open sets in $G$  and, since $ |\partial G_k'|=0$,
    \[
 G =(\bigcup_{k=0}^5 G_k')\cup N, \quad |N|=0.
 \]
  
 For each $\bar y\in G,$  the point $Y=(A,B)=(Du (\bar y),\partial_t  v(\bar y))$ belongs to $\Sigma(\lambda');$ thus  $Y\in S_{i'}(\lambda'')$ for some $1\le i' \le 5$ and $0\le\lambda''<\lambda'.$ In particular, if $\bar y\in G_k$ for some $1\le k\le 5$,  then we set  $i'=k$ and $\lambda''=\lambda.$ With this notation, we write $Y=\lambda''\xi_{i'}(q')+(1-\lambda'')\pi_{i'}(q')$ for some $q' \in B_{\beta}.$  Let  $X_j=\mu\xi_j(q')+(1-\mu)\pi_j(q')\in S_j(\mu)$ for all $1\le j\le 5.$  
 We  apply Theorem \ref{lem1} to $Y$ and $\{X_j\}$ with $G$ being the cube $Q$  to obtain  $ (\varphi, \psi)\in C^\infty_c(Q;\R^2\times\MM)$ such that
\begin{itemize}
\item[(i)] $\dv\psi=0$ and $(A+D\varphi, B+\partial_t \psi)\in \Sigma(\mu')$ on $\bar Q;$  
\item[(ii)]  ${\int_{Q_0} \varphi(x,t)\,dx=0}$ for all $t\in (-1,1);$ 
\item[(iii)]  $\|\varphi\|_{L^\infty(Q)}+\|\partial_t \varphi\|_{L^\infty(Q)}<\epsilon';$  
\item[(iv)]  there exist disjoint open  subsets $P_1,\cdots, P_5$ of $Q$ such that 
\begin{equation}\label{eq-d2}
\begin{cases} (A+D\varphi, B+\partial_t\psi) \big |_{P_j} =X_j \quad \forall\, 1\le j\le 5,\\
|P_{i'}|>(1-\epsilon') \left (\frac{\lambda''}{\mu} +(1-\frac{\lambda''}{\mu}) \tilde \nu_{i'}^{i'}(\mu,q)\right ) |Q|,\\
|P_j|> (1-\epsilon')   ( 1-\frac{\lambda''}{\mu}) \tilde \nu_{i'}^j(\mu,q) \,|Q| \quad\forall\, j\ne i'.
\end{cases} 
\end{equation} 
\end{itemize}
Note that $\sum_{j=1}^5 |P_j|> (1-\epsilon')|Q|.$ Moreover, since  $\lambda''<\lambda'$, by (\ref{eq-1}), 
\begin{equation}\label{eq-e}
\begin{cases}  
|P_{i'}|>(1-\epsilon') \left (\frac{\lambda''}{\mu} +(\mu-\lambda')(\mu-\nu_1)^4\nu_0\right ) |Q|,\\
|P_j|> (1-\epsilon')  (\mu-\lambda')(\mu-\nu_1)^4\nu_0 \,|Q| \quad\forall\,1\le j\le 5.
\end{cases} 
\end{equation}

Given any $0<l<\epsilon$, let 
$
 (\tilde \varphi, \tilde \psi)= \mathcal L_{\bar y, l} (\varphi, \psi) $ and $  \tilde  g=\mathcal R_{\bar y,l}  \varphi$ be the rescaled functions  on $\bar Q_{\bar y,l}$ as defined in Section \ref{s3}.  Define
\begin{equation}\label{new-fun0}
\tilde u  =u_{\bar y,l}=  u +\tilde \varphi,\;\; \tilde  v  =v_{\bar y,l}=  v + \tilde \psi + \tilde g \; \; \mbox{ on $\bar Q_{\bar y,l}.$}
 \end{equation}
 Then, $\tilde u \in  u +C_c^\infty(Q_{\bar y,l};\R^2)$, $\tilde  v \in W^{1,\infty}_{v}(Q_{\bar y,l};\MM)\cap C^1(\bar Q_{\bar y,l})$, and 
 \[
\begin{cases}
  \tilde u= \dv \tilde  v,\\
  \|\tilde u-u\|_{L^\infty(Q_{\bar y,l})}+\|\partial_t \tilde   u -\partial_t  u\|_{L^\infty(Q_{\bar y,l})} < \epsilon',\\
 \|\partial_t \tilde  g\|_{L^\infty(Q_{\bar y,l})}\le  C_0 l \|\partial_t \varphi\|_{L^\infty(Q)}<C_0 l.
\end{cases}
\]
Note that
\[
\begin{split} &\Big |(D\tilde  u(y),\partial_t \tilde v(y))-\big (A+D\varphi(\frac{y-\bar y}{l}), B+\partial_t\psi(\frac{y-\bar y}{l})\big ) \Big |\\
&\le |D u(y)-D u(\bar y)|+|\partial_t  v(y)-\partial_t  v(\bar y)|+|\partial_t \tilde g(y)|  \mbox{\;  on  $  \bar Q_{\bar y,l}.$} 
\end{split}
\]

By the  continuity of $(Du,\partial_t v)$  and the openness of sets $G'_k, \, \Sigma(\mu')$ and $S_j(\mu)$, we select a sufficiently small $l_{\bar y} \in (0,\epsilon)$ such  that  for all $0<l<l_{\bar y},$
\begin{equation}\label{cubes}
 \begin{cases}  \bar Q_{\bar y,l}\subset G'_k \;\; \mbox{if $\bar y\in G'_k$ for some $0\le k\le 5,$} \\
 (D u, \partial_t v)\big |_{\bar Q_{\bar y,l}} \in B_\epsilon(Y), \;\;  (D\tilde  u, \partial_t \tilde  v)\big |_{\bar Q_{\bar y,l}} \in \Sigma(\mu'),  \\
  (D\tilde  u, \partial_t \tilde  v)\big |_{\tilde P_j} \in  S_j(\mu)\cap B_\epsilon (X_j) \quad  \forall\, 1 \le j \le 5,\end{cases}
\end{equation}
where $\tilde P_j$'s  are the disjoint open sets of $Q_{\bar y,l}$ obtained by rescaling $P_j.$  Hence, 
\begin{equation}\label{pro-d}
\begin{cases}  
| Q_{\bar y,l}\cap \{(D\tilde u,\partial_t\tilde v)\in K(\mu)\}|\ge |\cup_{j=1}^5 \tilde P_j|>(1-\epsilon')\,|Q_{\bar y,l}|,\\
|\tilde P_{i'}|>(1-\epsilon') \left (\frac{\lambda''}{\mu} +(\mu-\lambda')(\mu-\nu_1)^4\nu_0 \right ) |Q_{\bar y,l}|,\\
|\tilde P_j|> (1-\epsilon')  (\mu-\lambda')(\mu-\nu_1)^4\nu_0\,|Q_{\bar y,l}| \quad\forall\,1\le j\le 5.
\end{cases} 
\end{equation}

For each $0\le k\le 5$ with $G_k\ne \emptyset,$ the nonempty open set $G_k'$ is covered by the family of closed cubes 
$\{\bar Q_{\bar y,l}\, |\;\bar y\in G'_k,\; 0<l<l_{\bar y}\}$ in the sense of   {\em Vitali covering}  (see \cite{D}); thus, we have
 \[
G'_k = \big(\bigcup_{n=1}^{m_k}\bar Q^k_n \big)\cup N_k,\quad 1\le m_k\le \infty,\;\; |N_k|=0,
 \]
 where the cubes $\bar Q^k_n=\bar Q_{\bar y_n^k,l_n^k}\subset G'_k$ are such that  $\bar y_n^k\in G'_k,$ $0<l_n^k=\rad(Q_n^k)<\epsilon$ and $\bar Q_n^k\cap\bar Q_m^k=\emptyset$ for  all $n\ne m.$ 
 Clearly if $G_j\ne \emptyset$ and $j\ne k$ then $\bar Q_n^k\cap\bar Q_m^j=\emptyset$ for all $1\le n \le m_k$  and $1\le m\le m_j$.
 We thus achieve  an at most countable partition of $G$ by disjoint closed cubes:
  \begin{equation}\label{near-null}
G=\big (\bigcup_{k=0}^5\bigcup_{n=1}^{m_k} \bar Q^k_n \big )\cup N,\quad |N|=0,
 \end{equation}
where we set $m_k=0$ if $G_k=\emptyset$ and drop the $k$-terms from the union.

For each $0\le k\le 5$ with $G_k\ne\emptyset,$  let $\tilde  u^k _{n}=  u _{\bar y^k_n,l^k_n}$ and $\tilde  v^k_n = v_{\bar y^k_n,l^k_n}$ be the functions defined by (\ref{new-fun0}) on $\bar Q^k_n=\bar Q_{\bar y^k_n,l^k_n}$ with center $\bar y^k_n\in  G'_k$ and $\rad(Q^k_n)<\epsilon.$ Define
\begin{equation}\label{final-fun}
(\tilde  u,\tilde v)  =    \sum_{k=0}^5\sum_{n=1}^{m_k} (\tilde  u^k_n, \tilde v^k_n) \chi_{\bar Q^k_n}+(u,v)\chi_N,
\end{equation}
where  again no $k$-terms are in  the summation  if $m_k=0.$ 
Then  $(\tilde u,\tilde v) \in W_{(u,v) }^{1,\infty}(G;\R^2\times \MM)$ and $(\tilde u,\tilde v) \big |_{\bar Q^k_n} \in C^1(\bar Q^k_n;\R^2\times \MM)$ for all $0\le k\le 5$ and $1\le n\le m_k.$ We define the cubes $\{\bar Q_j\}_{j=1}^M$ by re-indexing: 
\[
\{\bar Q_n^k\,|\; 0\le k\le 5,\; 1\le n\le m_k\}=\{\bar Q_j \,|\; 1\le j\le M\}.
\]

 It is easily seen that the requirements (a)-(c) are satisfied; the requirement (d) also follows readily from the first line of (\ref{pro-d}). 
 
 To satisfy  the first line of the requirement (e), note that  by the third line of (\ref{pro-d}) we have
 \[
  |\{(D\tilde u,\partial_t \tilde v)\in S_j(\mu)\}|  \ge (1-\epsilon')  (\mu-\lambda')(\mu-\nu_1)^4\nu_0 \sum_{k=0}^5 \sum_{n=1}^{m_k} |Q_n^k|\]
  \[
 = (1-\epsilon')    (\mu-\lambda')(\mu-\nu_1)^4\nu_0  |G| 
  \]
   \[
> \frac12 (\mu-\lambda')(\mu-\nu_1)^4\nu_0  |G|, 
 \]
which holds   if   the number $\epsilon'\in (0,\epsilon)$ is  chosen to  satisfy $ \epsilon'  <1/2.$  

The second  line of the requirement (e) is automatically satisfied if $G_k=\emptyset.$ Now assume $1\le k\le 5$ with  $G_k\ne \emptyset.$ In this case, we have   $i'=k$ and $\lambda''=\lambda$ in (\ref{eq-e}) and thus, by (\ref{set-G'}) and the second line of  (\ref{pro-d}) 
 \[
  |\{(D\tilde u,\partial_t \tilde v)\in S_k(\mu)\}|    \ge   \sum_{n=1}^{m_k}  | \bar Q_n^k \cap \{(D\tilde u^k_n,\partial_t\tilde v^k_n)\in S_k(\mu)\}|  \]
  \[
  > (1-\epsilon') \left (\frac{\lambda}{\mu} +(\mu-\lambda')(\mu-\nu_1)^4\nu_0\right ) \sum_{n=1}^{m_k} |Q_n^k|\]
  \[
  = (1-\epsilon') \left (\frac{\lambda}{\mu} +(\mu-\lambda')(\mu-\nu_1)^4\nu_0 \right )  |G_k'|\]
  \[
  \ge (1-\epsilon')^2 \left (\frac{\lambda}{\mu} +(\mu-\lambda')(\mu-\nu_1)^4\nu_0 \right )  |G_k|. 
 \]
Thus the second line of the requirement (e) is ensured  if   the number $\epsilon'\in (0,\epsilon)$ is chosen to further satisfy
 \begin{equation}\label{small1}
 (1-\epsilon')^2 \left (\frac{\lambda}{\mu} +(\mu-\lambda')(\mu-\nu_1)^4\nu_0 \right ) >\frac{\lambda}{\mu},
 \end{equation}
   which is possible because  $(\mu-\lambda')(\mu-\nu_1)^4\nu_0>0.$ Therefore, both requirements in  (e) are  ensured. 
 
  Finally, to verify the requirement (f),  note that 
 \begin{equation}\label{f0}
 \begin{split}
& \|D\tilde u-Du\|_{L^1(G)} =\sum_{k=0}^5  \|D\tilde u-Du\|_{L^1(G_k)} 
\\
 &\le C|G_0| + \sum_{k=1}^5  \|D\tilde u-Du\|_{L^1(G_k)} 
 \\
 &\le C|G_0| + \sum_{k=1}^5 C |G_k\setminus G'_k|  +\sum_{k=1}^5  \|D\tilde u-Du\|_{L^1(G'_k)} 
 \\
 &
 \le C(|G_0| + \epsilon'|G|)  + \sum_{k=1}^5\sum_{n=1}^{m_k} \int_{Q^k_n} |(D\tilde u,\partial_t \tilde v)-(D u,\partial_t  v)| 
 \\
 &
 \le C(|G_0| + \epsilon |G|)  + \sum_{k=1}^5\sum_{n=1}^{m_k}\int_{\tilde P^k_{n}}|(D\tilde u,\partial_t \tilde v) -(D u,\partial_t  v)| 
 \\
& \qquad\qquad \qquad \qquad + \sum_{k=1}^5\sum_{n=1}^{m_k} \int_{Q^k_n\setminus \tilde P^k_{n}}|(D\tilde u,\partial_t \tilde v)-(D u,\partial_t  v)|, \end{split}
\end{equation}
  where  $\tilde P^k_{n}\subset Q_n^k=Q_{\bar y^k_n,l^k_n}$ for all $1\le k\le 5$ and  $1\le n\le m_k$ are the sets defined as in (\ref{cubes}) with $Y=Y_n=(Du(\bar y_n),\partial_t v(\bar y_n))\in S_k(\lambda)$ and with $i'=k$ and $\lambda''=\lambda$ in  (\ref{pro-d}). By (\ref{small1}),
  \[
  (1-\epsilon')  \left (\frac{\lambda}{\mu} +(\mu-\lambda')(\mu-\nu_1)^4\nu_0  \right ) >  \frac{\lambda}{\mu}, 
  \]
  and thus $|\tilde P^k_{n}|>\frac{\lambda}{\mu}| Q_n^k|.$  Hence
 \begin{equation}\label{f2}
  \int_{Q^k_n\setminus \tilde P^k_{n}}  |(D\tilde u, \partial_t \tilde v) -(D u,\partial_t  v)|  \le C|Q^k_n\setminus \tilde P_n^k|<C(\mu-\lambda)|Q^k_n|.
\end{equation}
  Let $X_k=X^k_{n}\in S_k(\mu)$ be defined as in (\ref{cubes}) with  $Y=Y_n$  by some $q_n\in B_{\beta}$. Then 
  $
   |X_n^k-Y_n|=(\mu-\lambda)|\xi_k(q_n)-\pi_k(q_n)|\le C(\mu-\lambda);$  thus, by (\ref{cubes}) and (\ref{pro-d}),  we have
\begin{equation}\label{f1}
 \begin{split}
 \int_{\tilde P_n^k}  |(D\tilde u,&\partial_t \tilde v)  -(D u,\partial_t  v)| \\
 &  \le  \int_{\tilde P^k_n} \left (|(D\tilde u,\partial_t \tilde v)-X_n^k|+|X_n^k-Y_n|+|(D u,\partial_t  v)-Y_n|\right )\\
 & \le C(\epsilon+(\mu-\lambda))|Q^k_n|. \end{split}
  \end{equation}
Plugging (\ref{f2}) and (\ref{f1}) into (\ref{f0}) proves the requirement (f).  

The proof is completed.  
\end{proof}


  \section{Proof of the main theorem}  
 Given a function $f\in L^\infty(G;\R^q)$ on an open set $G\subset \R^d$,  for any $x_0\in G$, we define the {\em essential oscillation} of $f$ at $x_0$ by
  \[
  \omega_f(x_0)=\inf\{\|f(x)-f(y)\|_{L^\infty(U\times U)}\; |\; U\subset G, \;  \mbox{$U$  open, } \;  x_0\in  U \}.
  \]
  We say that $f$ is {\em essentially continuous at $x_0$} if $\omega_f(x_0)=0.$ Clearly, if $f$ is continuous at $x_0$ then it must be essentially continuous at $x_0,$  but  the converse is false, as easily  seen  by the Dirichlet function on $\R.$

In this final section, we prove the following more general theorem.   

\begin{thm}\label{mainthm1}  Assume the polyconvex function $F$ is defined as above.  Let  $(\bar u, \bar v)\in C^1(\bar \Omega_T;\R^2\times\MM)$ satisfy  
\begin{equation}\label{subs}
\bar u=\dv \bar v, \quad (D\bar u,\partial_t \bar v)\in \Sigma(\bar \lambda) \;\;\mbox{  
 on $\bar \Omega_T$}
 \end{equation}
   for some $0< \bar \lambda<1.$  Then, for each $\rho>0,$  the  problem 
\begin{equation}\label{ibvp-5}
\begin{cases}  \partial_t u=\dv DF(Du )   \;\; \mbox{\rm on $\Omega_T;$}
\\
 u  |_{\partial\Omega_T}=\bar u 
\end{cases}
\end{equation}  possesses a  Lipschitz weak solution   $u$ such that  $\|u-\bar u\|_{L^\infty}+\|\partial_t u-\partial_t \bar u\|_{L^\infty} <\rho$  and that $Du$ is not essentially continuous at any point; in fact,  $\omega_{Du}(y_0)\ge \delta$ for all $y_0\in\Omega_T,$ where is $\delta>0$ is a uniform constant.
  \end{thm}

\begin{remk} (i) This result is a type of  {\em $h$-principle} for the gradient flow (\ref{GF}) in the sense of {\sc Gromov} \cite{Gr86}.

(ii) The theorem asserts that the weak solutions exist in every $L^\infty$ neighborhood of $\bar u;$ therefore, such solutions must be infinitely many.  

(iii)   By a {\em global}  weak solution  of (\ref{GF}) we mean a Sobolev function $u$ on $\Omega_\infty=\Omega\times (0,\infty)$ such that  (\ref{weak-sol}) holds for all $\varphi\in C^1([0,\infty); C^\infty_c(\Omega;\R^m))$ and all $T>0.$  

(iv)  Since the boundary condition in Theorem \ref{mainthm1} is  a full Dirichlet one on the whole boundary  $\partial\Omega_T,$  by glueing solutions on finite time steps, we have the following {\em global existence} result.
\end{remk}

\begin{thm}
Let $(\bar u,\bar v)$ be a $C^1$ strict subsolution of (\ref{pdr1}) on $\bar \Omega_\infty.$ Then,   for each $\rho>0,$ there exists a Lipschitz global  weak solution  $u$ of the gradient flow (\ref{GF}) on $\Omega_\infty$  such that $\|u-\bar u\|_{L^\infty}+\|\partial_t u-\partial_t \bar u\|_{L^\infty} <\rho$, $u|_{\partial'\Omega_\infty}=\bar u$ and  $Du$ is nowhere essentially continuous on $\Omega_\infty.$
\end{thm}

\begin{proof}[Proof of Theorem \ref{mainthm}] Before proving Theorem \ref{mainthm1}, we show that it implies our main theorem as a  corollary. For example,  let $\phi\in C^1_0(\Omega;\R^2)$ be given as in Theorem \ref{mainthm}.  Let $J\subset \R^2$ be a rectangle containing $\Omega$ and extend $\phi$ to $J$ by zero. Let $h=\mathcal R^J \phi\in C^1(\bar J;\MM)$  and define
 \begin{equation}\label{sub}
 \bar u(x,t)= \frac{\epsilon}{T} \phi(x) t,\quad \bar v(x,t)=\frac{\epsilon}{T} h(x) t.
 \end{equation}
Then, $(\bar u,\bar v)\in C^1(\bar\Omega_T;\R^2\times\MM)$, $\bar u=\dv \bar  v$ on $\bar\Omega_T$, $\bar u|_{\partial'\Omega_T}=0$  and 
\[
\|(D\bar u,\partial_t  \bar v)-(0,0)\|_{L^\infty}  \le  |\epsilon| (\|D\phi\|_{L^\infty}+\frac{1}{T}\|h\|_{L^\infty}).
\]
Since $
(0,0)=\pi_1(0)\in \pi_1(B_\beta)=S_1(0)$, which  is  open in $\MM\times \MM,$  we select  a number $\epsilon_0>0$ depending only on $\phi$  such that  \begin{equation}\label{suff-0}
 (D\bar  u, \, \partial_t \bar  v) \in S_1(0)\subset \Sigma(\bar \lambda) \mbox{ \; on $\,\bar\Omega_T$}\quad \forall\; |\epsilon|<\epsilon_0,
\end{equation} 
where $0<\bar \lambda<1$ is any fixed number. Then, for all $|\epsilon|<\epsilon_0$,  with the strict subsolution $(\bar u,\bar v)$  defined by (\ref{sub}), Theorem \ref{mainthm} follows from Theorem \ref{mainthm1}.
\end{proof}

 \begin{proof}[Proof of Theorem \ref{mainthm1}]   First we select sequences $\{\lambda_n\}$, $\{\lambda'_n\}$ and $\{\epsilon_n\}$  such that
\begin{equation} 
 \begin{cases} \lambda_0=0, \quad  \lambda_1> \max\{\bar\lambda,\, \delta_1,\,  \nu_1\},\\
 \lambda_{n+1}>\lambda_{n}, \quad \lim_{n\to\infty}\lambda_n=1;\\
\lambda_0'=\bar\lambda,\quad 
 \lambda_n'= \frac12(\lambda_{n}+\lambda_{n+1}) \quad \forall\, n=1,2,\cdots; \\
   \epsilon_n=\rho/3^n. \end{cases}
   \end{equation}
   
Since $\{S_i(0)\}_{i=1}^5$ is a family of disjoint open sets, we can apply Theorem \ref{thm1} to  $(u,v)=(\bar u,\bar v)$ on the set $G=\Omega_T$ with
\[
 \lambda=\lambda_0,\; \;  \lambda'=\lambda_0',\; \; \mu=\lambda_1,\;\; \mu'=\lambda_1',\;\;\epsilon=\epsilon_1
 \]
 to obtain a function $(u_1,v_1)= (\tilde u,\tilde v)\in W^{1,\infty}_{(\bar u,\bar v)}(\Omega_T;\R^2\times \MM)$ that is piece-wise $C^1$ on at most  countable partition of $\Omega_T$ by disjoint closed cubes $\{\bar Q^1_j\}_{j=1}^{M_1}$ ($M_1\le \infty$)  such that
\begin{equation}\label{u1}
\begin{cases}
 0<\rad(Q^1_j)<\epsilon_1;\\ 
 u_1=\dv v_1, \; (D u_1,  \partial_t v_1)\in \Sigma(\lambda_1')  \mbox{ on each $\bar Q^1_j$;}\\ 
 \|u_1-\bar u\|_{L^\infty(\Omega_T)}+ \|\partial_t u_1-\partial_t  \bar u\|_{L^\infty(\Omega_T)}<\epsilon_1;\\
 |Q^1_j\cap \{(D u_1,\partial_t v_1) \in K(\lambda_1) \}| >(1-\epsilon_1)\,|Q^1_j| \quad \forall\,1\le j\le M_1.
 \end{cases}
 \end{equation}
 
 We now construct  $(u_n,v_n)$ for all $n\ge 2$ by induction. Suppose that for a fixed $n\ge 1$ we have constructed $(u_n,v_n) \in W^{1,\infty}_{(\bar u,\bar v)}(\Omega_T;\R^2\times \MM)$ that is piece-wise $C^1$ on at most  countable partition of $\Omega_T$ by disjoint closed cubes $\{\bar Q^n_j\}_{j=1}^{M_n}$ ($M_n\le \infty$).

 Since $\lambda_n>\delta_1$, by Property (P3)(c), $\{S_i(\lambda_n)\}_{i=1}^5$ is a family of disjoint  open sets, and thus we can  apply Theorem \ref{thm1} to  $(u,v)=(u_n,v_n)$ on the set  $G= Q_j^n$ for each $1\le j\le M_n$ with
\[
 \lambda=\lambda_n,\;\; \lambda'=\lambda_{n}',\;\; \mu=\lambda_{n+1},\;\; \mu'=\lambda_{n+1}',\;\;\epsilon=\epsilon_{n+1},
 \]
 to obtain a function  $(\tilde u^{n,j},\tilde v^{n,j})\in W^{1,\infty}_{(u_n,v_n)}(Q^n_j;\R^2\times \MM)$ that is piece-wise $C^1$ on at most  countable partition of $Q^n_j$ by disjoint closed cubes $\{\bar Q^{n,j}_i\}_{i=1}^{M_{n,j}}$ ($M_{n,j}\le \infty$)  such that 
\begin{equation}\label{eq55}
\begin{cases}
 0<\rad(Q_i^{n,j})<\epsilon_{n+1} \quad \forall \, i; \\
 \tilde u^{n,j}=\dv \tilde v^{n,j}, \; (D\tilde u^{n,j},  \partial_t \tilde v^{n,j})\in \Sigma(\lambda_{n+1}')    \mbox{ on  $\bar Q_i^{n,j}$} \quad \forall \,i;  \\ 
 \|\tilde u^{n,j}-u_n\|_{L^\infty(Q^n_j)}+ \|\partial_t \tilde u^{n,j} -\partial_t  u_n\|_{L^\infty(Q^n_j)}<\epsilon_{n+1}; \\  
 |Q_i^{n,j}\cap  \{(D\tilde u^{n,j}, \partial_t \tilde v^{n,j}) \in K(\lambda_{n+1}) \}| >(1-\epsilon_{n+1})\,|Q_i^{n,j}| \quad\forall \, i;\\
 \|D\tilde u^{n,j}-Du_n\|_{L^1(Q^n_j)}\le C\big[\epsilon_n+\epsilon_{n+1}+ (\lambda_{n+1}- \lambda_n)\big ]|Q^n_j|;\\
 |Q^n_j\cap  \{( D\tilde u^{n,j} , \partial_t \tilde v^{n,j})  \in S_k(\lambda_{n+1}) \}|  \\  
 \;\; >\frac14  (\lambda_{n+1}  - \lambda_n) (\lambda_{n+1}-\nu_1)^4 \nu_0\,|Q^n_j| \quad  \forall\,  1\le k\le 5; \\
 |Q^n_j\cap\{( D\tilde u^{n,j} ,\partial_t \tilde v^{n,j})  \in S_k(\lambda_{n+1}) \}|  \\  
 \;\; >\frac{\lambda_n}{\lambda_{n+1}}  |Q^n_j\cap\{(D u_n,\partial_t  v_n) \in S_k(\lambda_n) \}| \quad \forall\,  1\le k\le 5.  
 \end{cases}
 \end{equation}
 Here,   the fifth line of (\ref{eq55})  follows from  the part (f) of Theorem \ref{thm1} with the set $G_0=Q_j^n\cap\{(Du_n,\partial_t v_n)\notin K(\lambda_n)\},$ which  satisfies  $|G_0|  <  \epsilon_n  |Q_j^n|$ by the induction assumption:
 \[
 |Q_j^n\cap\{(Du_n,\partial_t v_n)\in K(\lambda_n)\} |>(1-\epsilon_n)|Q^n_j| 
 \]
 (see the last line of (\ref{u1}) for $n=1$).  

We glue all cubes  $\{Q_j^n\}_{j=1}^{M_n}$ together to have a partition of $\Omega_T$ with
 \[
 \Omega_T=(\bigcup_{j=1}^{M_n}\bar Q_j^n)\cup N_n= (\bigcup_{j=1}^{M_n} Q_j^n)\cup \tilde N_n,\quad \tilde N_n=(\bigcup_{j=1}^{M_n} \partial Q_j^n)\cup N_n,
 \]
where $|N_n|=|\tilde N_n|=0.$ We then  define a  function $(u_{n+1},v_{n+1})\in W^{1,\infty}(\Omega_T;\R^2\times \MM)$ by setting
 \begin{equation}\label{next-fun}
 (u_{n+1},v_{n+1})=\sum_{j=1}^{M_n} (\tilde u^{n,j},\tilde v^{n,j})\chi_{Q_j^n} +(u_n,v_n)\chi_{\tilde N_n}.
 \end{equation}
 The family of  closed cubes $\{\bar Q^{n+1}_h\}_{h=1}^{M_{n+1}}$ is defined simply by re-indexing:
 \[
 \{\bar Q^{n,j}_i\,|\; 1\le j\le M_n,\; 1\le i\le M_{n,j}\} =\{\bar Q^{n+1}_h \,|\; 1\le h \le M_{n+1}\}.
 \]
Then it is easily seen that 
\[
(u_{n+1},v_{n+1})\in W^{1,\infty}_{(\bar u,\bar v)}(\Omega_T;\R^2\times \MM),
\]
  is piece-wise $C^1$ on at most  countable partition of $\Omega_T$ by disjoint closed cubes $\{\bar Q^{n+1}_h\}_{h=1}^{M_{n+1}}$ ($M_{n+1}\le \infty$),  and satisfies  the following properties: 
\begin{eqnarray}
&&0<\rad(Q^{n+1}_h)<\epsilon_{n+1} \;\; \forall \, h; \\
&&u_{n+1}=\dv v_{n+1}, \; (Du_{n+1},  \partial_t v_{n+1})\in \Sigma(\lambda_{n+1}')    \mbox{ on each $\bar Q^{n+1}_h;$} \label{eq57}\\ 
&&\|u_{n+1}-u_n\|_{L^\infty(\Omega_T)}+ \|\partial_t u_{n+1} -\partial_t  u_n\|_{L^\infty(\Omega_T)}<\epsilon_{n+1}; \label{cauchy}\\  
&&  |\{(Du_{n+1}, \partial_t v_{n+1}) \in K(\lambda_{n+1}) \} | >(1-\epsilon_{n+1})\,|\Omega_T|; \label{eq59}\\
&&\|Du_{n+1}-Du_n\|_{L^1(\Omega_T)}\le C\big[\epsilon_n+\epsilon_{n+1}+ (\lambda_{n+1}- \lambda_n)\big ]|\Omega_T|;\label{cauchy1}\\
&&|Q^n_j\cap  \{( Du_{n+1}, \partial_t v_{n+1})  \in S_k(\lambda_{n+1}) \}|  \label{eq510} \\
&&> \frac14  (\lambda_{n+1}  - \lambda_n) (\lambda_{n+1}-\nu_1)^4\nu_0\,|Q^n_j| \;  \forall\, 1\le j \le M_n, \;1\le k\le 5; \nonumber\\
 &&|Q^n_j\cap\{( Du_{n+1},\partial_t v_{n+1})  \in S_k(\lambda_{n+1}) \}|  \label{eq511} \\
&&>\frac{\lambda_n}{\lambda_{n+1}} |Q^n_j\cap\{(D u_n,\partial_t  v_n) \in S_k(\lambda_n) \}| \;  \forall\, 1\le j \le M_n, \;1\le k\le 5. \nonumber
 \end{eqnarray}
 
\begin{lem}\label{lem512} For  all $n>m\ge 1$ and $1\le j\le M_m$,  we have
 \[
 |Q^m_j  \cap\{( Du_{n+1} ,\partial_t v_{n+1})  \in S_k(\lambda_{n+1}) \}| \]
 \[
 >\frac14 \frac{\lambda_{m+1}}{\lambda_{n+1}}  (\lambda_{m+1}- \lambda_m)(\lambda_{m+1}-\nu_1)^4\nu_0\,|Q^m_j|,
  \]
  for all $1\le k\le 5.$
 \end{lem}
\begin{proof} By the construction of sequence $\{(u_{n},v_{n})\}_{n=1}^\infty$ and the cubes $\{Q^{n}_j\}$ above, we  see that  if $n>m\ge 1$ and $1\le j\le M_m$ then $Q^m_j=(\cup_{k\in I} Q^n_k)\cup N$ for some index set $I$ and  a null-set $N.$ 
Thus,  for all $n>m\ge 1$, $1\le j\le M_m$ and $1\le k\le 5,$   by (\ref{eq510}) and  (\ref{eq511}) we have  
 \[
 |Q^m_j  \cap\{( Du_{n+1} ,\partial_t v_{n+1})  \in S_k(\lambda_{n+1}) \}| \]
 \[
 >\frac{\lambda_n}{\lambda_{n+1}} |Q^m_j\cap\{(Du_n,\partial_t  v_n) \in S_k(\lambda_n) \}|\]
 \[
 >\frac{\lambda_n}{\lambda_{n+1}} \cdot \frac{\lambda_{n-1}}{\lambda_{n}}|Q^m_j\cap\{(Du_{n-1},\partial_t  v_{n-1}) \in S_k(\lambda_{m+1}) \}| \]
 \[
 >\frac{\lambda_n}{\lambda_{n+1}}\cdot \frac{\lambda_{n-1}}{\lambda_{n}} \cdot \cdots \cdot \frac{\lambda_{m+1}}{\lambda_{m+2}}|Q^m_j\cap\{(Du_{m+1},\partial_t  v_{m+1}) \in S_k(\lambda_{m+1}) \}| \]
 \[
 =\frac{\lambda_{m+1}}{\lambda_{n+1}} |Q^m_j\cap\{(Du_{m+1},\partial_t  v_{m+1}) \in S_k(\lambda_{m+1}) \}| \]
 \[
 >\frac14 \frac{\lambda_{m+1}}{\lambda_{n+1}}  (\lambda_{m+1}- \lambda_m)(\lambda_{m+1}-\nu_1)^4\nu_0\,|Q^m_j|. 
 \]
 \end{proof}

\begin{lem} \label{lem52} Let $f(X)= DF(A)-B $ for $X=(A,B)\in \MM\times \MM.$
   Then, there is a constant $C>0$ such that 
   \[
  | f(X)|\le C\,(1-\lambda) \quad \forall\; 0\le \lambda\le 1,\; X\in K(\lambda).
   \]  
   \end{lem}
   \begin{proof}
   Since $F$ is smooth, $f$ is locally Lipschitz and thus $|f(X_1)-f(X_2)|\le  L|X_1-X_2|$ for all $X_1,X_2\in \overline{\Sigma(1)}$, where  $L>0$ is a constant. Let $X\in K(\lambda)$. Then $X=\lambda\xi_i(q)+(1-\lambda)\pi_i(q)$ for some $1\le i\le 5$ and $q\in B_\beta.$ Note that $\xi_i(q)\in \K$ and thus $f(\xi_i(q))=0$; hence $|f(X)|\le L|X-\xi_i(q)|=L(1-\lambda)|\xi_i(q)-\pi_i(q)|\le C(1-\lambda).$
   \end{proof}
    
    We are now in a position to complete the proof. 
    
    By (\ref{cauchy}) and (\ref{cauchy1}), $\{u_n\}$ is a Cauchy sequence in $W^{1,1}(\Omega_T;\R^2).$ Let $u_n\to u$ in $W^{1,1}(\Omega_T;\R^2)$ as $n\to\infty.$ Clearly, $u$ also belongs to $W^{1,\infty}_{\bar u}(\Omega_T;\R^2).$ Moreover, by (\ref{cauchy}), we have
 \[
 \begin{split}
 \|u  -\bar u\|_{L^\infty(\Omega_T)}  + &\| \partial_t u -\partial_t  \bar u\|_{L^\infty(\Omega_T)}\\
  \le \sum_{n=0}^\infty [\|u_{n+1}-u_n\|_{L^\infty(\Omega_T)}  &+ \|\partial_t u_{n+1} -\partial_t  u_n\|_{L^\infty(\Omega_T)}] \\
    \le \sum_{n=0}^\infty {\epsilon_{n+1}} &=\frac12 \rho <\rho. \end{split}
\]

By (\ref{eq59})  and Lemma \ref{lem52},  we have
\[
 \int_{\Omega_T} |f(Du_{n+1}, \partial_t v_{n+1})|   \le  C[(1-\lambda_{n+1})   +\epsilon_{n+1} ]|\Omega_T|;
 \]
thus $\|f(Du_{n+1}, \partial_t v_{n+1})\|_{L^1(\Omega_T)} \to 0$   as $n\to \infty.$ 
 From  $u_{n+1}=\dv v_{n+1}$ a.e.\,on $\Omega_T$,  it follows easily that 
  \begin{equation}\label{weak-id}
\begin{split}  &\int_\Omega    (u_{n+1}  \cdot  \varphi)  \big |_{t=0}^{t=T}    - \int_{\Omega_T} \big (u_{n+1} \cdot \partial_t \varphi  - DF(Du_{n+1}) : D\varphi \big )\\  = \int_{\Omega_T} &\big (DF(Du_{n+1})-\partial_t v_{n+1}\big ):D\varphi \quad \forall\, \varphi\in C^1([0,T];C^\infty_c(\Omega;\R^2)). \end{split}
 \end{equation}
Since $u_{n+1}\to u$ in $W^{1,1}(\Omega_T)$, $f(Du_{n+1}, \partial_t v_{n+1}) \to 0$ in $L^1(\Omega_T)$ both strongly as $n\to \infty$ and $\{\|Du_{n+1}\|_{L^\infty(\Omega_T)}\}$ is bounded, taking the limits in (\ref{weak-id}), we have that
 \[
 \int_\Omega  (u  \cdot \varphi)   \big |_{t=0}^{t=T}  - \int_{\Omega_T} \big (u  \cdot \partial_t \varphi  - DF(Du) : D\varphi \big ) =0 
\]
holds for all $\varphi\in C^1([0,T];C^\infty_c(\Omega;\R^2)).$ This proves that $u$ is a weak solution of (\ref{ibvp-5}).

 Finally, to show $Du$ is nowhere essentially continuous on $\Omega_T,$ let $y_0\in \Omega_T$ and $U$ be any open subset of $\Omega_T$ containing $y_0.$  From \[
  U\cap(\cap_{n=1}^\infty \cup_{j=1}^{M_n} \bar Q^n_j) \ne \emptyset 
  \]
 and   $\rad(Q^n_j)<\epsilon_n\to 0$, it follows  that there exist some $m\ge 1$ and $1\le j\le M_m$ such that $\bar Q^m_j\subset U.$ 
 
 Recall  the projection operator: $\mathbb P(A,B)=A.$ By Lemma \ref{lem512},   we have for all $n>m$ and $1\le k\le 5$,
\[
 |Q^m_j\cap \{Du_{n+1}  \in \mathbb P(S_k(\lambda_{n+1}))\}| \]
 \[
  \ge  |Q^m_j  \cap\{( Du_{n+1} ,\partial_t v_{n+1})  \in S_k(\lambda_{n+1}) \}| \]
  \[
 >\frac14 \frac{\lambda_{m+1}}{\lambda_{n+1}}  (\lambda_{m+1}- \lambda_m) (\lambda_{m+1}-\nu_1)^4\nu_0\,|Q^m_j|
 \] 
 Taking the limits as $n\to \infty$, since $u_{n+1}\to u$ in $W^{1,1}(\Omega_T;\R^2)$, we have
\[
 |Q^m_j\cap \{Du \in \overline{\mathbb P(S_k(1))}\}  \ge\frac14\lambda_{m+1}(\lambda_{m+1}-\lambda_m)(\lambda_{m+1}-\nu_1)^4\nu_0\, |Q^m_j|>0 
\]
for all $1\le k\le 5.$  Let
\[
\delta=\min_{k\ne l}\dist(\overline{\mathbb P(S_k(1))};\overline{\mathbb P(S_l(1))}).
\]
By Property (P1), we have $\delta>0$. Thus $\|Du(y)-Du(z)\|_{L^\infty(U\times U)}\ge \delta>0$ for all open  sets $U\subset\Omega_T$ containing $y_0$; this proves  \[
\omega_{Du}(y_0)\ge \delta>0.
\] 

The proof of Theorem \ref{mainthm1} is now completed.
  \end{proof}
  
  \section*{Data Availability Statement}
Data sharing not applicable to this article as no datasets were generated or analyzed during the current study.

\end{document}